\documentclass[11pt,a4paper]{amsart}
\usepackage[colorlinks=true, pdfstartview=FitV, linkcolor=blue, 
citecolor=blue]{hyperref}

\usepackage{amssymb,bbm,amscd}
\usepackage{graphicx}
\usepackage{a4wide}

\usepackage{color}

\theoremstyle{plain}
\newtheorem{theorem}{Theorem}[section]

\newtheorem{lemma}[theorem]{Lemma}

\theoremstyle{definition}
\newtheorem{remark}[theorem]{Remark}

\numberwithin{equation}{section}
 
\newcommand{\dd}{\,\mathrm{d}}
\newcommand{\ts}{\hspace{0.5pt}}

\newcommand{\defeq}{\mathrel{\mathop:}=}

\DeclareMathOperator{\card}{\mathrm{card}}

\newcommand{\vL}{\varLambda}
\newcommand{\ZZ}{\mathbb{Z}}
\newcommand{\QQ}{\mathbb{Q}}
\newcommand{\RR}{\mathbb{R}\ts}

\newcommand{\NN}{\mathbb{N}}
\newcommand{\XX}{\mathbb{X}}
\newcommand{\exend}{\hfill $\Diamond$}
\newcommand{\rad}{\textnormal{rad}}

\newcommand{\myfrac}[2]{\frac{\raisebox{-2pt}{$#1$}}
      {\raisebox{0.5pt}{$#2$}}}

\begin{document}

\title[Scaling of $k$-free diffraction]{Scaling of the
diffraction measure of \\[2mm]
$k$-free integers near the origin}

\author{Michael Baake} 
\address{Fakult\"{a}t f\"{u}r Mathematik,
  Universit\"{a}t Bielefeld,\newline \hspace*{\parindent}Postfach
  100131, 33501 Bielefeld, Germany}
\email{mbaake@math.uni-bielefeld.de}

\author{Michael Coons}
\address{School of Mathematical and Physical Sciences,
University of Newcastle, 
\newline \hspace*{\parindent}University Drive, Callaghan NSW 2308, Australia}
\email{michael.coons@newcastle.edu.au}

\begin{abstract}
  Asymptotics are derived for the scaling of the total diffraction
  intensity for the set of $k$-free integers near the origin, which
  is a measure for the degree of patch fluctuations.
\end{abstract}

\maketitle
\thispagestyle{empty}

\section{Introduction}

Given a point set $\vL$ in Euclidean $d$-space, an immediate natural
question to ask is that of the existence of its density, and further,
if the density does exist, how it fluctuates locally. More generally,
one can consider point patches and define means and variances for
their appearance. When $\vL$ is a lattice, there is no fluctuation (in
the sense that each patch repeats lattice-periodically), while a
random set such as the positions of an ideal gas (as modelled by a
Poisson process) shows a lot of fluctuation. What intermediate regimes
exist is a natural aspect of spatial order, and such considerations
have lately garnered much interest. 

More recently, due to general progress in the theory of aperiodic
order, see \cite{TAO} and references therein for background,
fluctuation or variance considerations have been extended to such
systems, both in the projection realm and in the context of inflation
systems \cite{Soc1,Soc2,BG-scaling}. The key to many of these
investigations is understanding the scaling behaviour of the
diffraction measure $\widehat{\gamma}$, which is the mathematical
counterpart of the \emph{structure factor} used in physics, in the
vicinity of the origin (in reciprocal space).  In one dimension, the
scaling behaviour of
$Z(x) \defeq \widehat{\gamma} \bigl( (0,x]\bigr)/\ts \widehat{\gamma}
\bigl( \{ 0 \} \bigr)$ as $x\to 0^+$ is an important tool for this
\cite{TS,BG-scaling}. Indeed, one can interpret a decay rate
$Z(x)/x\to 0$, as $x\to 0^+$, of the diffraction intensity as
fluctuations (as $R\to\infty$ in direct space) of the point set and
its patches in density; see \cite{Luck} for early results on aperiodic
examples. Within this framework, Torquato and Stillinger \cite{TS}
introduced and investigated the notion of \emph{hyperuniformity},
referring to point patterns that do not possess infinite-wavelength
fluctuations.

While one gets $Z(x) = x$ for the homogeneous Poisson process of unit
density in $\RR$, power laws of the form $Z(x) \sim x^{\alpha}$ with
$\alpha > 1$, as $x\to 0^+$, are typical for aperiodically ordered
sets \cite{Luck, Soc1,Soc2}; in this way, these sets are hyperuniform. 
Systems such as the Thue{\ts}--Morse chain show a decay that is
faster than any power \cite{GL,BG-scaling}, and a lattice displays a
function $Z(x)$ that drops down to $0$ for sufficiently small $x>0$,
thus signifying a perfect (in fact, periodic) repetition of patches of
any size. In this way, the statistic $Z(x)$ allows for various notions 
of intermediate order; for a nice discussion on this topic, see the 
introduction of Brauchart, Grabner and Kusner \cite{BGK}, as well
as \cite{BG-scaling} for a summary of results.

Naturally, one can also address such questions for point sets of
number-theoretic origin, such as the square-free integers on the line
and their various generalisations; see \cite{BH,PH,BHS} and references
therein for systematic examples. Here, we consider the set $V_k$ of
$k$-free integers with fixed $k\geqslant 2$, that is, the elements of
$\ZZ$ that are not divisible by the $k$-th power of any (rational) prime
number. The sets $V_k$ are examples of weak model sets of maximal
density, and are thus approachable by the projection method
\cite{BHS,Keller}. The diffraction measure of $V_k$ is known to be a
pure point measure \cite{BMP}, which is explicitly computed in terms
of elementary number-theoretic functions. So, we can investigate the
function $Z_k (x)$ for this family of point sets in $\ZZ\subset \RR$.
With the knowledge of the diffraction measure, which will be recalled
below, our task can be termed as follows.

For any integer $q\ne 0$, let $\bar{q}\defeq\rad(q)$ denote the
square-free part of $ q $, that is, its square-free divisor of
largest modulus. For $q, k\in\NN$ with $k\geqslant 2$, set
\begin{equation}\label{eq:def-fk}
  f^{}_k (q) \, \defeq \,\begin{cases}
    \prod_{p|\bar{q}}\frac{1}{p^k-1} , &\mbox{if
      $q$ is $(k+1)$-free,}\\ 0 , &\mbox{otherwise.}\end{cases}
\end{equation}
Clearly, $f^{}_{k} (1)=1$, and $f^{}_{k} (q) = f^{}_{k} (\bar{q})$ 
when $q$ is $(k+1)$-free.  Also, write
\begin{equation}\label{eq:def-genphi}
  \varphi(x,q) \, \defeq \,
  \card\{1\leqslant m\leqslant qx: \gcd(m,q)=1\} 
\end{equation}
for the $2$-parameter totient function, where both $x>0$ and $q\in\NN$
are assumed. Note that $\varphi (x,q) =0$ whenever $qx<1$.  In this
article, we are interested in analysing the asymptotics, as
$x\to0^+$, of
\begin{equation}\label{eq:Z-sum}
    Z^{}_k(x) \, = \sum_{\substack{q\in\NN_{k+1}\\ q\geqslant 1/x}}
    \varphi(x,q) f^{}_k (q)^2 ,
\end{equation}
where $\NN_{j}$ with $j\geqslant 2$ denotes the set of $j$-free
positive integers; that is, $\NN_{j}=V_j\cap\NN$.

As our main result, we shall establish the following asymptotic.

\begin{theorem}\label{thm:main}
  Let\/ $k\geqslant 2$ be a fixed positive integer. As\/
  $x \to 0^{+}$, we have
\[ 
  \log \bigl(Z^{}_k(x)\bigr) \, \sim \,
  \left(2 - \myfrac{1}{k}\right) \log (x) \ts .
\]
\end{theorem}

From here on, the article is organised as follows. In
Section~\ref{sec:prelim}, we recall the essential results on the
diffraction measure of the $k$-free integers and then state some
results of elementary or asymptotic nature that we require for the
proof of Theorem~\ref{thm:main} in Section~\ref{sec:proof}. Finally,
we briefly comment on further developments in
Section~\ref{sec:further}.

\section{Preliminaries and known results}\label{sec:prelim}

Let us begin by recalling some results from \cite{BMP} on the diffraction
of $k$-free integers. It is well known that the set $V_k$ has natural
density $1/\zeta(k)$, which can be obtained as a limit along any
averaging sequence of growing intervals around an arbitrary, but fixed
centre. This is called \emph{tied density} in \cite{BMP}. If any such
averaging sequence is fixed, say $\bigl( [-n,n]\bigr)_{n\in\NN}$ for
instance, also all patch frequencies exist; in particular, the
$2$-point correlations within $V_k$ exist.  If $\eta^{}_{k} (m)$ is
the frequency of occurrence of two points within $V_k$ at distance
$m \in\ZZ\setminus \{0\}$, and $\eta^{}_{k} (0) = 1/\zeta(k)$ is the
density of $V_k$, one obtains the \emph{autocorrelation measure} of
$V_k$ as
\[
     \gamma^{\phantom{j}}_{k} \, = \sum_{m\in\ZZ}
     \eta^{\phantom{j}}_{k} (m) \, \delta^{}_{m} \ts ,
\]
where $\delta_x$ denotes the normalised Dirac measure at $x$; compare
\cite[Ch.~9 and Sec.~10.4]{TAO}.

The autocorrelation measure is positive definite, and hence Fourier
transformable as a measure \cite[Prop.~8.6]{TAO};
$\widehat{\ts\gamma^{}_{k}\ts}$ is known as the \emph{diffraction
  measure} of $V_k$. The main result of \cite{BMP} on
$\widehat{\ts\gamma^{}_{k}\ts}$ can be summarised as follows.

\begin{theorem}
   The diffraction measure\/ $\widehat{\ts\gamma^{}_{k}\ts}$
   of\/ $V^{}_k$ is a pure point measure, supported on
\[
    L^{\circledast}_{k} \, \defeq \, \big\{ y \in \QQ :
     y = \tfrac{m}{q} \text{ with }
     m\in \ZZ, \, q\in\NN^{}_{k+1} \text { and }
     \gcd (m,q) =1  \big\} ,
\] 
   which is a subgroup of\/ $\QQ$. More precisely, the diffraction 
   measure reads
\[
    \widehat{\ts\gamma^{}_{k}\ts} \, = \sum_{z\in L^{\circledast}_{k}}
    I^{}_{k} (z) \, \delta^{}_{z} \ts ,
\]   
   where, for any\/ $z = \frac{m}{q} \in L^{\circledast}_{k}$, the intensity
   is given by
\[
     I^{}_{k} ( z )  \, = \, \left( \frac{f^{}_{k} (q)}{\zeta (k)}\right)^{2}
\]
     with the function\/  $f^{}_{k} (q)$ as defined in Eq.~\eqref{eq:def-fk}. 
     \qed
\end{theorem}

Note that $0=\frac{0}{1}\in L^{\circledast}_{k}$, with
$I^{}_{k} (0) = \zeta(k)^{-2} = \widehat{\ts \gamma^{}_{k} \ts} \bigl(
\{ 0 \} \bigr)$.  Clearly, the diffraction measure is reflection
symmetric with respect to the origin, which means that the scaling
near $0$ can be determined from the positive side. Here, for $x>0$,
one gets
\[
  Z^{}_{k} (x) \, \defeq \, \frac{\widehat{\ts \gamma^{}_{k} \ts}
    \bigl( (0,x] \bigr)}{\widehat{\ts \gamma^{}_{k} \ts}
    \bigl( \{ 0 \} \bigr)}\, = \, \zeta(k)^2 \!
    \sum_{\substack{z\in L^{\circledast}_{k}\\
    0 < z \leqslant x}} I^{}_{k} (z) \, = 
    \sum_{q\geqslant 1/x} \sum_{\substack{1\leqslant m \leqslant qx \\
    \gcd (m,q) =1}} f^{}_{k} (q)^2 ,
\]
which leads immediately to Eq.~\eqref{eq:Z-sum}. Now, according to the 
definition of $\varphi (x,q)$, without changing the
value of $Z_k(x)$, we can enlarge the summation set in \eqref{eq:Z-sum}
to include all $q\in\NN_{k+1}$ subject to the weaker condition
$\bar{q}^k \geqslant 1/x$. This is possible since, whenever
$q\in\NN_{k+1}$, we have both $q\leqslant \bar{q}^k$ and
$\varphi (x,q) =0$ for all terms with $q<1/x$; the latter statement is
true because the interval $[1,qx]$ is empty under this condition.
Consequently, we have
\begin{equation}\label{eq:Z-alt}
    Z^{}_k (x) \, = 
    \sum_{\substack{q\in\NN_{k+1}\\ \bar{q}^k \geqslant 1/x}}
    \varphi(x,q) f^{}_k (q)^2 .
\end{equation}

\begin{remark}
  The set of $k$-free integers defines a topological dynamical system
  $(\XX_k, \ZZ)$, where $\XX_k$ is the orbit closure of $V_k$ under
  the natural, continuous shift action of $\ZZ$ in the local
  topology. The frequency measure $\nu$ with respect to the chosen
  averaging sequence is well defined and known as the \emph{Mirsky
    measure}. It is ergodic, and $V_k$ is a generic set for it; see
  \cite{BHS} and references therein for details.  The
  measure-theoretic dynamical system $(\XX_k, \ZZ, \nu )$ has pure
  point dynamical spectrum, where the latter (in additive notation) is
  precisely the Abelian group $L^{\circledast}_{k}$ from above. This
  follows from \cite{BMP} via the equivalence theorem on pure point
  spectra \cite{BL}; compare \cite{BHS}. For the case $k=2$, it has
  also been derived by explicit means in \cite{CS}.  Alternatively, it
  systematically follows from Keller's new approach
  \cite{Keller,Keller2}.  \exend
\end{remark}

To continue, we require the following preliminary results from
\cite{Apo,HW}, where, for $n \in \NN$, we use $d(n)$ to denote the
number of positive divisors of $n$, $\sigma(n)$ for the sum of those
divisors, and $\varphi(n)$ for Euler's totient function, that is,
$\varphi(n)=\varphi(1,n)$ with the function from
\eqref{eq:def-genphi}.

\begin{lemma}\label{lem:fk}
  For any integers\/ $q,k\geqslant 2$ with $q\in \NN_{k+1}$, we have
\[  
  f^{}_k (q) \, = \, \frac{1}
  {\varphi(\bar{q})\, \sigma(\bar{q}^{\ts k-1})}
  \qquad \text{and} \qquad
  \frac{1}{\bar{q}^{\ts k}} \, < \,
  f^{}_k(q) \, < \, \frac{\zeta (k)}{\bar{q}^{\ts k}} \ts .
\]
\end{lemma}

\begin{proof}
  Under our conditions, the equality
  $f^{}_k(q)=\bigl(\varphi(\bar{q})\, \sigma(\bar{q}^{\ts
    k-1})\bigr)^{-1}$ follows easily from the standard definitions of
  $\varphi$ and $\sigma$. Then, following the argument of
  \cite[Thm.~329]{HW}, we have
\[
    \sigma(\bar{q}^{\ts k-1}) \, = \,
    \prod_{p|\bar{q}}\frac{p^{k}-1}{p-1}
    \, = \, \bar{q}^{\ts k-1}\prod_{p|\bar{q}}
    \frac{1-p^{-k}}{1-p^{-1}} \ts .
\]
Since $\varphi(\bar{q})=\bar{q}\prod_{p|\bar{q}}(1-p^{-1})$ and
$\zeta(k) = \prod_{p} (1-p^{-k})^{-1}$, we clearly get
\[
  {\bar{q}^{\ts k}} {f^{}_k (q)} \, = \,
  \frac{\bar{q}^{\ts k}} 
  {\varphi(\bar{q})\ts\ts \sigma(\bar{q}^{\ts k-1})}
  \, = \, \prod_{p|\bar{q}}
  \frac{1}{1-p^{-k}}\, \in \, 
  \bigl( 1, {\zeta(k)} \bigr) ,
\]
  which is the desired result.
\end{proof}

\begin{lemma}\label{lem:phi}
  Let\/ $x>0$ be fixed, $q\in\NN^{}_{k+1}$, and let\/ $\ell=q/\bar{q}.$
  Then, one has
\[
  \varphi(x,q) \, = \, \varphi(\ell x, \bar{q}) \ts .
\]
Moreover, for square-free\/ $q$ --- which means\/
$q = \bar{q} \in\NN^{}_{2}$ --- we have
\[
  \left|\varphi(x,q) - x \ts\ts \varphi(q)\right|\leqslant  d(q) .
\]
Consequently, we have the inequality 
\[
\varphi(x,q) \, \leqslant \, x \ts\ts \varphi(q) + d(q).
\]
\end{lemma}

\begin{proof} 
  The first claim follows from the definition of $\varphi(x,q)$ in
  \eqref{eq:def-genphi}, since $\gcd(m,q)=1$ if and only if
  $\gcd(m,\bar{q})=1$, where $q = \ell\ts \bar{q}$ with
  $\ell \ts | \ts \bar{q}^k$ under our assumption. The second and third
  relations follow from \cite[Secs.~1 and 2]{CN2005}.
\end{proof}

\begin{lemma}\label{lem:mualpha}
  Let\/ $k\geqslant 2$ be an integer. For sufficiently small positive 
  real $y$, one has
\[
  \sum_{q\geqslant {y^{-1}}}
  \frac{\lvert\mu(q)\rvert}{q^{k}} \, = \,
  \frac{y^{k-1}}{(k-1) \ts \zeta(2)} +
  O\bigl( y^{k-\frac{1}{2}}\bigr) \, = \,
  \frac{y^{k-1}}{(k-1) \ts \zeta(2)}
  \bigl( 1 + O (y^{1/2})\bigr). 
\]  
\end{lemma}

\begin{proof}
Let $M(t) \defeq \sum_{q\leqslant t} \lvert \mu (q) \rvert$.
By partial summation, we obtain 
\[
  \sum_{q\geqslant {y^{-1}}}
  \frac{\lvert \mu(q)\rvert}{q^{k}} \, = \, k
  \int_{y^{-1}}^\infty M(t) \ts
  \frac{\dd t}{t^{k+1}} \, - \, y^{k}
  \sum_{q< {y^{-1}}} \lvert \mu(q) \rvert \ts .
\]  
Since $M(t) = \frac{t}{\zeta(2)} + O \bigl(\sqrt{t}\,\bigr)$ 
for $t\in[1,\infty)$, see \cite[Thm.~333]{HW}, we find
\begin{align*}
  \sum_{q\geqslant {y^{-1}}}
     \frac{\lvert\mu(q)\rvert}{q^{k}}
  \, & = \, k\int_{{y^{-1}}}^\infty
       \left(\frac{t}{\zeta(2)} +
       O \bigl(\sqrt{t}\, \bigr)\right)
       \frac{\dd t}{t^{k+1}} \, - \,
       y^{k}\left(\frac{y^{-1}}{\zeta(2)}
       + O \bigl( y^{-1/2} \bigr)\right)\\
     & = \, \frac{k}{\zeta(2)}
       \int_{y^{-1}}^\infty \frac{\dd t}{t^{k}}
       \, - \, \frac{y^{k-1}}{\zeta(2)}
       \, + \, O\! \left(\int_{y^{-1}}^\infty
       \frac{\dd t}{t^{k+\frac{1}{2}}}\right)
       +O \bigl( y^{k-\frac{1}{2}}\bigr)\\[2mm]
     & = \, \frac{y^{k-1}}{(k-1) \ts \zeta(2)}
       \, + \, O\bigl( y^{k-\frac{1}{2}}\bigr). \qedhere
\end{align*}
\end{proof}
We now have all prerequisites to turn to our main result.

\section{Proof of Theorem~\ref{thm:main}}\label{sec:proof}

We remind the reader that $k\geqslant 2$ is a fixed integer, thus any
dependence on $k$ will be suppressed. Our arguments use a real
parameter $\alpha\in[1,k]$, whose value will be specified later as a
function of another parameter $\varepsilon$. In what follows, since we
are ultimately interested in a limit as $x\to0^+$, all big-$O$ terms
are taken over $x\in(0,x_0(k)]$ for some sufficiently small
$x_0(k)>0$. Note also that the constants in the big-$O$ symbols to
follow do depend on $k$, but are independent of $\alpha$ in the range
$\alpha\in[1,k]$.

Recall that a $k$-free $q$ has a unique decomposition
$q=\ell\ts \bar{q},$ where $\bar{q}$ is squarefree and
$\ell \ts | \ts \bar{q}^{\ts k-1}$, so that
$q\leqslant \bar{q}^{\ts k}$. By Eq.~\eqref{eq:Z-alt} and two
applications of Lemma \ref{lem:phi},
\begin{align}
\nonumber  Z^{}_k(x) \, & = \sum_{\substack{q\in\NN_{k+1}\\
    \bar{q}^{\ts k}\ts \geqslant 1/x}}
    \varphi(x,q) f^{}_k (q)^2  \; = 
    \sum_{\substack{\tilde{q}\in\NN^{}_{2}\\ {\tilde{q}^{\ts k}}\geqslant 1/x}}
    f^{}_k (\tilde{q})^2 \sum_{\ell|\tilde{q}^{\ts k-1}}
     \varphi(\ell x,\tilde{q}) \\[2mm]
\nonumber    & \geqslant
    \sum_{{q^\alpha}\geqslant 1/x}|\mu(q)|\,
  f^{}_k (q)^2 \sum_{\ell|q^{k-1}}\varphi(\ell x,q) \\[2mm]
\nonumber  & = \sum_{q^\alpha \geqslant 1/x}
    |\mu(q)|\, f^{}_k (q)^2 \sum_{\ell|q^{k-1}} \Bigl( \ell x \ts \varphi(q)
    + O\bigl( d(q)\bigr) \Bigr).
\end{align}
Therefore, for $q^\alpha\geqslant \frac{1}{x^{}_0(k)}$, we have, using
Lemma~\ref{lem:fk} for the last two lines,
\begin{align}
  \nonumber Z^{}_k(x) \,  & \geqslant \, x \!
    \sum_{{q^\alpha}\geqslant 1/x}|\mu(q)|\,
    \varphi(q)\ts\ts \sigma(q^{k-1})f^{}_k (q)^2 \, + \,
    O\! \left(\sum_{{q^\alpha}\geqslant 1/x}
    \! |\mu(q)|\,d(q) \ts\ts d(q^{k-1}) f^{}_k (q)^2 \right)\\
\nonumber & = \, x \! \sum_{q^\alpha \geqslant 1/x}\lvert \mu (q) \rvert \ts
  f^{}_{k} (q) \, + \, O \! \left( \sum_{q^\alpha \geqslant 1/x}
  \lvert \mu (q) \rvert \ts\ts d(q) \ts \ts d(q^{k-1})
  f^{}_{k} (q)^2 \right) \\[1mm]
\label{ps1}  &\geqslant \, x  \!  \sum_{\substack{{q}\geqslant x^{-1/\alpha}}}
    \frac{\lvert \mu(q) \rvert}{q^k} \, + \, O \! \left(
    \sum_{ {q}\geqslant x^{-1/\alpha}}
    \frac{\lvert \mu(q)\rvert \ts\ts d(q) \ts\ts
    d(q^{k-1})}{q^{2k}} \right),
\end{align}
where, as stated above, the implied constant in the big-$O$ term
depends on the value of $k$, but is independent of $\alpha$ in the
range $1\leqslant \alpha\leqslant k$.

Now let $\varepsilon\in(0,1)$ be arbitrary, but fixed. Using a result 
recorded in Hardy and Wright \cite[Thm.~317]{HW}, for $x$ 
small enough, and thus, correspondingly, for $q$ large enough,
one has
\begin{equation}\label{eq:HW}
    d(q) \ts \ts d(q^{k-1}) \, \leqslant \, 
    q^{\varepsilon }.
\end{equation}
Consequently, we also have, for $x\leqslant x_0(k)$,
\[
    Z^{}_k(x) \, \geqslant \, x  \!  
    \sum_{\substack{{q}\geqslant x^{-1/\alpha}}}
    \! \frac{\lvert \mu(q) \rvert}{q^k} \; + \; O\! \left(
    \sum_{ {q}\geqslant x^{-1/\alpha}}
    \frac{\lvert \mu(q)\rvert}{q^{2k-\varepsilon}} \right).
\]

Applying Lemma \ref{lem:mualpha} with $y=x^{1/\alpha}$, and using 
only the leading term, we have 
\[
   \sum_{ {q}\geqslant x^{-1/\alpha}}
   \frac{\lvert \mu(q)\rvert}{q^{2k-\varepsilon}}
    \, = \, O_\varepsilon\!\left({x^{
    \frac{1}{\alpha}(2k-1-\varepsilon)}}\right),
\]
where $O_\varepsilon$ is used to indicate the dependence of the
implied constant on $\varepsilon$. Continuing the inequality for
$Z_k(x)$ from above, and applying Lemma \ref{lem:mualpha} yet again
with $y=x^{1/\alpha}$, we obtain
\begin{equation}\label{eq:OO}
     Z^{}_k(x) \, \geqslant \, 
     \frac{x^{\frac{k-1}{\alpha}+1}}{(k-1) \ts \zeta(2)} \, + \,
  O\left( x^{\frac{2 k-1}{2 \alpha}+1}\right) \, + \, O_\varepsilon\!\left({x^{
      \frac{1}{\alpha}(2k-1-\varepsilon)}}\right). 
\end{equation}

As before, the implied constants in the big-$O$ terms in Eq.~\eqref{eq:OO} are independent of the parameter $\alpha\in[1,k]$. Therefore, we are free to choose $\alpha$ in that range, and in order to make the second remainder term asymptotically smaller than the first term on the right, we must have $\frac{k-1}{\alpha}+1<\frac{1}{\alpha}(2k-1-\varepsilon)$, which simplifies to 
\begin{equation}\label{eq:allx}
    \alpha \, < \, k-\varepsilon \ts .
\end{equation} Thus we let  $\alpha=k(1-\varepsilon)$, so that $\alpha$ satisfies \eqref{eq:allx}. 
Now, \eqref{eq:OO} becomes
\begin{equation*}
  Z^{}_k(x) \, \geqslant \, \frac{\ts x^{\left(2-\frac{1}{k}\right)
  \left(\frac{1}{1-\varepsilon}\right)-
  \frac{\varepsilon}{1-\varepsilon} }}
  {(k-1) \ts \zeta(2)}  \, + \,  O\!\left( x^{\left(2-\frac{1}{2k}\right)
   \left(\frac{1}{1-\varepsilon}\right)-\frac{\varepsilon}{1-\varepsilon}}\right)
   +O_\varepsilon\!\left( x^{\left(2-\frac{1}{k}\right)
   \left(\frac{1}{1-\varepsilon}\right)-\frac{1}{k}
   \left(\frac{\varepsilon}{1-\varepsilon}\right)}\right). 
\end{equation*}
Since the final two big-$O$ terms have, for fixed $\varepsilon$, larger 
exponents than the main term, we obtain 
\begin{equation}\label{eq:liminf}
    \liminf_{x\to 0^+}\, \frac{\log \bigl(Z^{}_k(x)\bigr)}{\log (x)} 
    \, \geqslant \, 2 - \myfrac{1}{k}\ts .
\end{equation}
  
For an upper bound on the limit superior of $\log 
\bigl(Z^{}_k(x)\bigr)/{\log (x)},$ as $x\to 0^+$, we employ the 
last assertion of Lemma~\ref{lem:phi}, which states that
$\varphi(x,q) \leqslant x\ts \varphi(q)+d(q)$ for positive
square-free integers $q$. As in the above argument, we 
let $\varepsilon\in(0,1)$ be arbitrary, but fixed, so that, for 
$x$ sufficiently close to zero, starting with Eq.~\eqref{eq:Z-alt} 
and applying Lemma~\ref{lem:fk}, we get
\begin{align}
\nonumber  Z^{}_k(x) \, & = \, x \! \sum_{{q^k}\geqslant 1/x}|\mu(q)|\,
    \varphi(q)\ts\ts \sigma(q^{k-1})f^{}_k (q)^2 \; + 
    \sum_{{q^k}\geqslant 1/x}
    \! |\mu(q)|\,d(q) \ts\ts d(q^{k-1}) f^{}_k (q)^2 \\[2mm]
\nonumber & = \, x \!
   \sum_{q^k \geqslant 1/x}\lvert \mu (q) \rvert \ts
  f^{}_{k} (q) \; +  \sum_{q^k \geqslant 1/x}
  \lvert \mu (q) \rvert \ts\ts d(q) \ts \ts d(q^{k-1})
  f^{}_{k} (q)^2 \\[2mm]
\label{ps2} &\leqslant \, x \, \zeta(k) \!
     \sum_{\substack{{q}\geqslant x^{-1/k}}}
    \frac{\lvert \mu(q) \rvert}{q^k} \; +  \sum_{ {q}\geqslant x^{-1/k}}
    \frac{\lvert \mu(q)\rvert}{q^{2k-\varepsilon}}\ts ,
\end{align}
where for the last step, as above, we have used 
the inequality \eqref{eq:HW}.

Applying Lemma~\ref{lem:mualpha}, with $y=x^{1/k}$, to each of the sums in \eqref{ps2}, for $x$ sufficiently close to zero, we have 
\begin{align*} 
      Z^{}_k (x) &\, \leqslant \, \frac{x^{\frac{k-1}{k}+1}
      \zeta(k)}{(k-1) \ts \zeta(2)} \, + \,
     O\!\left( x^{\frac{2 k-1}{2 k}+1}\right) \,+ \,
     \frac{x^{\frac{1}{k}(2k-1-\varepsilon)}}{(k-1)\ts \zeta(2)} \,+ \,
     O_\varepsilon\!\left(x^{2 - \frac{1+\varepsilon}{2k}}\right)\\[2mm]
      &= \, \frac{x^{2-\frac{1+\varepsilon}{k}}}{(k-1)\ts \zeta(2)}
         \, \bigl( 1 + \zeta(k) \ts x^{\varepsilon} \bigr)
        \,+ \, O\!\left( x^{2-\frac{1}{2 k}}\right) \,+ \, O_\varepsilon\!\left(x^{2 - \frac{1+\varepsilon}{2k}}\right), 
\end{align*}
where, as before, we have written $O_\varepsilon$ to indicate the dependence of the implied constant on $\varepsilon$. Since the final two big-$O$ terms have, for fixed $\varepsilon$, larger exponents than the main term, we obtain 
\begin{equation}\label{eq:limsup}
     \limsup_{x\to 0^+}\, \frac{\log \bigl(Z^{}_k(x)\bigr)}{\log (x)}  \, = \,
       2 - \myfrac{1}{k}\ts .
\end{equation}

Putting together \eqref{eq:liminf} and \eqref{eq:limsup} completes the proof 
of the theorem.

\section{Further developments}\label{sec:further}

The interested reader will note that the contributions from the  M\"obius 
and divisor sums can be made explicit using more delicate arguments from
analytic number theory. In fact, with the use of such deep techniques, 
after seeing this paper on the arXiv, N.~Rome and E.~Sofos have proved 
a power-law asymptotic for $Z_k(x)$ using methods quite different from 
ours; see their preprint \cite{RSpre} for details.

\section*{Acknowledgements}
It is our pleasure to thank Uwe Grimm for valuable discussions. We also 
thank the anonymous referee for a number of helpful suggestions, and 
N.~Rome and E.~Sofos for communicating with us about their result, 
which is now available on the arXiv (\texttt{arXiv:1907.04845}). Our
work was supported by the Research Centre for Mathematical Modelling
(RCM$^{2}$) of Bielefeld University.


\begin{thebibliography}{99}

\bibitem{Apo}
Apostol T M,
\textit{Introduction to Analytic Number Theory},
corr.\ $4\ts$th printing, Springer, New York (1995).

\bibitem{TAO}
Baake M and Grimm U,
\textit{Aperiodic Order. Vol.\ 1: A Mathematical Invitation},
Cambridge University Press, Cambridge (2013).

\bibitem{BG-scaling}
Baake M and Grimm U,
Scaling of diffraction measures near the origin: Some
rigorous results, \textit{J.\ Stat.\ Mech.: Th.\ Exp.}
\textbf{2019}, paper 054003 (25 pp).

\bibitem{BH}
Baake M and Huck C,
Ergodic properties of visible lattice points,
\textit{Proc.\ Steklov Inst.\ Math.} \textbf{288}
(2015) 165--188; \texttt{arXiv:1501.01198}. 
  
\bibitem{BHS}
Baake M, Huck C and Strungaru N,
On weak model sets of extremal density,
\textit{Indag.\ Math.} \textbf{28} (2017) 3--31;
\texttt{arXiv:1512.07129}.

\bibitem{BL}
Baake M and Lenz D,
Dynamical systems on translation bounded measures:\
Pure point dynamical and diffraction spectra,
\textit{Ergodic Th.\ \& Dynam.\ Syst.} \textbf{24} (2004) 1867--1893;
\texttt{arXiv:math.DS/0302061}.

\bibitem{BMP}
Baake M, Moody R V and Pleasants P A B,
Diffraction for visible lattice points and $k$th power
free integers, \textit{Discr.\ Math.} \textbf{221} (2000) 3--42; 
\texttt{arXiv:math.MG/9906132}.

\bibitem{BGK}
Brauchart J S, Grabner P J and Kusner W,
Hyperuniform point sets on the sphere:\ Deterministic aspects,
\textit{Constr.\ Approx.} \textbf{50} (2019) 45--61;
\texttt{arXiv:1709.02613}.

\bibitem{CS}
Cellarosi F and Sinai Y G,
Ergodic properties of square-free numbers,
\textit{J.\ Europ.\ Math.\ Soc.} \textbf{15} (2013)
1343--1374; \texttt{arXiv:1112.4691}.

\bibitem{CN2005}
Codec\`{a} P and Nair M,
The lesser known $\Delta$-function in number theory,
\textit{Amer.\ Math.\ Monthly} \textbf{112} (2005) 131--140.

\bibitem{GL}
Godr\`{e}che C and Luck J M,
Multifractal analysis in reciprocal space and the nature of
the Fourier transform of self-similar structures,
\textit{J.\ Phys.\ A:\ Math.\ Gen.} \textbf{23} (1990) 
3769--3797.

\bibitem{HW}
Hardy G M and Wright E M,
\textit{An Introduction to the Theory of Numbers},
6th ed., revised by Heath-Brown D R and Silverman J H,
Oxford University Press, Oxford (2008).  

\bibitem{Keller}
Keller G,
Maximal equicontinuous generic factors and weak model sets, 
\textit{Discr.\ Cont.\ Dynam.\ Syst.\ A}, in press;
\texttt{arXiv:1610.03998}.

\bibitem{Keller2}
Keller G,
Generalized heredity in $\mathcal{B}$-free systems,
\textit{preprint} \texttt{arXiv:1704.04079}.

\bibitem{Luck}
Luck J M,
A classification of critical phenomena on quasi-crystals and
other aperiodic structures, \textit{Europhys.\ Lett.}
\textbf{24} (1993) 359--364.

\bibitem{Soc1}
O\u{g}uz E C, Socolar J E S, Steinhardt P J and Torquato S,
Hyperuniformity of quasicrystals,
\textit{Phys.\ Rev.\ B} \textbf{95} (2017) 054119:1--10;
\texttt{arXiv:1612:01975}.

\bibitem{Soc2}
O\u{g}uz E C, Socolar J E S, Steinhardt P J and Torquato S,
Hyperuniformity and anti-hyperuniformity in one-dimensional
substitution tilings,
\textit{Acta Cryst.\ A} \textbf{75} (2019) 3--13;
\texttt{arXiv:1806.10641}. 

\bibitem{PH}
Pleasants A B P and Huck C,
Entropy and diffraction of the $k$-free points in
$n$-dimensional lattices,
\textit{Discr.\ Comput.\ Geom.} \textbf{50} (2013) 39--68;
\texttt{arXiv:1112.1629}.


\bibitem{RSpre}
Rome N and Sofos E,
On the diffraction measure of $k$-free integers,
\textit{preprint} \texttt{arXiv:1907.04845}.

\bibitem{TS}
Torquato S and Stillinger F H, 
Local density fluctuations, hyperuniformity, and order metrics, 
\textit{Phys.\ Rev.\ E} \textbf{68} (2003) 041113:1--25;
Erratum 069901; \texttt{arXiv:cond-mat/0311532}.

  
\end{thebibliography}
\end{document}